\documentclass[smallcondensed]{svjour3}                     
\smartqed  
\usepackage{graphicx}
\usepackage{amsmath}
\usepackage{amsfonts}

\begin{document}

\title{Boundary Behavior of Non-Negative Solutions of the Heat Equation in Sub-Riemannian Spaces\thanks{The author was supported in part by NSF
Grant DMS-0701001}
}


\author{Isidro H Munive
}


\institute{Isidro H Munive \at
            Department of Mathematics, Purdue University, 150 North University Street, West Lafayette, Indiana 47907-2067 \\
              Tel.:  (765) 494-1901\\
              Fax: (765) 494-0548\\
              \email{imunive@purdue.edu}           
}


\maketitle

\begin{abstract}
We prove Fatou type theorems for solutions of the heat equation in sub-Riemannian spaces. The doubling property of L-caloric measure, the Dahlberg estimate, the local comparison theorem, among other results, are established here. A backward Harnack inequality is proved for non-negative solutions vanishing in the lateral boundary. 
\keywords{Backward Harnack inequality \and Doubling condition \and Fatou theorems}
\subclass{MSC 35K10 \and MSC 35B05 \and 31B25 }
\end{abstract}

\section{Introduction}
\label{intro}
Fatou type theorems have been an interesting area of study since the appearance in 1968 of Hunt and Wheeden's paper on non-tangential convergence of harmonic functions in Lipschitz domains. In the 1980's Fatou theorems for second order elliptic equations experience a remarkable progress due, mainly, to the works of Caffarelli, et. al. in Lipschitz domains \cite{CFMS}, and Jerison and Kenig paper on NTA-domains \cite{JK}. It was until 1995 with the appearance of \cite{CG1} that Fatou theory was finally extended to sub-elliptic equations in divergence form. In that paper, the authors pointed out that due to the presence of characteristic points the right geometry in the sub-Riemannian case is the one given by $NTA$ domains, which are non-tangentially accessible domains with respect to the Carnot-Cartheodory metric.  
The purpose of this paper is to generalize the results obtained by Capogna and
Garofalo in \cite{CG1} to equations of parabolic type: 
\begin{equation}
L=-\sum_{i=1}^{m}X_{i}^{\ast }X_{i}-\frac{\partial }{\partial t},
\end{equation}
in a domain $D=\Omega \times \left( 0,T\right) $, $\Omega \subset \mathbb{R}%
^{n}$, where $X=\left\{ X_{1},...,X_{m}\right\} $ is $C^{\infty }\left(\mathbb{R}^{n}\right)$ and
satisfies H\"{o}rmander's finite rank condition
\begin{equation}
rank\text{ Lie}\left[ X_{1},...,X_{m}\right] =n
\end{equation}
for every $x\in \mathbb{R}^{n},$ with $\Omega $ an nontangentially accessible domain, $NTA$. Here $X^{\ast }_{i}$ denotes the formal adjoint of $X_{i}$. The main results in this work are the backward Harnack inequality (BHI), the doubling property of $L-caloric$ measure and the local comparison theorem. It is worth mentioning that, as it was first pointed out in the paper by Fabes, Garofalo and Salsa, \cite{FGS}, the boundary backward Harnack inequality and the doubling property are equivalent, hence it is enough to prove one of them. \\

Fatou type theorems in $NTA$ domains are important in the study of free boundary problems. In a recent paper by Danielli et al. \cite{DGP} the regularity of the free boundary was proved in the sub-elliptic case, in which the local comparison theorem played a crucial role. This indicates that it is worth to generalize Fatou theory to the sub-Riemannian setting in $NTA$ domains. \\

In this paper we exploit the relation between the Green function and the $L-caloric$ measure given by the Dahlberg theorem to obtain important estimates. This approach allows to prove results in a clear and elegant way. The organization of the paper is as follows. In section 2 we recall the definition of $NTA$ domain along with known results related with the operator $L$, such as Gaussian bounds for the fundamental solution and the Harnack inequality. In this section we also introduce the notion of the $L-caloric$ measure with the help of  the results in Bony's paper \cite{Bony}. \\

In section 3 we prove several basic estimates for non-negative solutions of $Lu=0$ that will be used throughout the paper. An example of such estimates is the Carleson lemma.  The proof of this estimate relies heavily in the H\"older 
continuity up to the boundary of solutions of $Lu=0$. In \cite{FGS} and \cite{G}, the H\"older continuity implies the fact that the $L-caloric$ measure is bounded away from zero near the boundary. In our present work we can see that the latter implies the former by suitably adapting a beautiful proof found in \cite{SY}.  \\

One of the main results in this paper is the backward Harnack inequality, or (BHI) in short. The (BHI) is crucial to the proof of the doubling condition since it permits to overcome the time gap in the parabolic version of Dahlberg's estimate. In \cite{FS}, Fabes and Safonov gave a quite ingenious proof of the backward Harnack inequality for parabolic equations in divergence form with time dependent coefficients. In section 4 we have been able to adapt their proof to the sub-Riemannian setting.  In that section we also show that an interior elliptic-type Harnack inequality is implied by the Carleson lemma.\\

The local comparison theorem is proved in section 5. Recall that in \cite{JK} the proof of such result was based in a complicated localization theorem due to Peter Jones, see \cite{Jones}. In \cite{CG1} the authors were able to prove the local comparison in the sub-elliptic setting thanks to an ingenious  idea of John Lewis. Their proof does not use Jones' localization theorem. In section 5 we generalize Lewis' idea  to the parabolic setting. \\
  
Finally, in \cite{CG1} the authors dealt with bounded measurable perturbations of sub-Laplacians, namely $\sum_{i=1}^{m}X_{i}^{\ast }\left(a_{ij}X_{i}\right)$, where $A=\left(a_{ij}\right)$ is an $m\times m$ matrix-valued function on $\mathbb{R}^{n}$, having $L^{\infty}$ entries, and satisfying for some $\lambda>0$ and for every $\xi\in\mathbb{R}^{m}$: $\lambda\left|\xi\right|^{2}\leq\sum_{i,j=1}^{m}a_{ij}\xi_{i}\xi{j}\leq\lambda^{-1}\left|\xi\right|^{2}$. So far proving Gaussian bounds and Harnack inequalities remains an open problem   for operators of the type 
\begin{displaymath}
L_{A}=-\sum_{i=1}^{m}X_{i}^{\ast }\left(a_{ij}\left(x,t\right)X_{i}\right)-\frac{\partial}{\partial t}.
\end{displaymath}
Once these results are established for $L_{A}$ the proofs given in this paper will apply without change to this more general setting.
\section{Preliminaries}
\label{sec:2}
Let $\Omega \subset \mathbb{R}^{n}$ be a domain and consider the
Carnot-Caratheodory distance associated to the family of vector fields $X=\left\{X_{1},...,X_{m}\right\}$, $d\left(\cdot,\cdot\right):\mathbb{R}^{n}\times\mathbb{R}^{n}\rightarrow\mathbb{R}^{+}$. In order to apply the results in Kosuoka and Stroock paper \cite{KS} we assume that the vector fields are bounded on $\mathbb{R}^{n}$. The metric balls will be
denoted by $B_{d}\left( x,r\right) =\left\{ y\mid d\left( x,y\right)
<r\right\} $.  In Nagel et al. \cite{NSW} the following result was obtained: there exist constants $C,R_{0}>0$, and a polynomial function with continuous coefficients $\Lambda\left(x,r\right)=\sum_{I}\left|a_{I}\left(x\right)\right|r^{d_{I}}$ such that for every $x\in\Omega$ and $r\leq R_{0}$
\begin{displaymath}
C\leq\frac{\left|B_{d}\left(x,r\right)\right|}{\Lambda\left(x,r\right)}\leq C^{-1}.
\end{displaymath}
The doubling property of the metric balls follows, namely
\begin{equation}
\left|B_{d}\left(x,2r\right)\right|\leq C_{1}\left|B_{d}\left(x,r\right)\right|
\label{doubling}
\end{equation}
for every $x\in\Omega$ and $r\leq R_{0}/2$.
We say that $B_{d}\left( x,r\right) $ is $\left( M,X\right)
-non-tangential$ $ball$ in $\Omega $ if 
\begin{equation*}
\frac{r}{M}<d\left( B_{d}\left( x,r\right) ,\partial \Omega \right) <Mr.
\end{equation*}%
For $x,y\in \Omega $, a $Harnack$ $chain$ from $x$ to $y$ in $\Omega $ is a
sequence of $\left( M,X\right) $-non-tangential balls in $\Omega $, $
B_{1},...,B_{p}$, with $x\in B_{1}$, $y\in B_{p}$, and $B_{i}\cap
B_{i+1}\neq \emptyset $ for $i=1,...,p-1$.\\

Now we can introduce $NTA$ domains.
\begin{definition}
We say that $\Omega$ is an nontangential accessible domain ($NTA$ domain)
if there exists M, $r_{0}>0$ for which:

\begin{enumerate}
\item (Interior corkscrew condition) For any $Q\in\partial\Omega$ and $r\leq
r_{0}$ there exists $A_{r}\left(Q\right)\in\Omega$ such that $\frac{r}{M}
\leq d\left(A_{r}\left(Q\right),Q\right)\leq r$ and $d\left(A_{r}\left(Q
\right),\partial \Omega\right)>\frac{r}{M}$.(This implies that $
B_{d}\left(A_{r}\left(Q\right),\frac{r}{2M}\right)$ is (3M,X)-nontangential.)

\item (Exterior corkscrew condition) $\Omega^{c} = \mathbb{R}
^n\setminus\Omega$ satisfies property (1).

\item (Harnack chain condition) For any $\epsilon>0$ and $x,y\in\Omega$ such
that $d\left(x,\partial \Omega\right)>\epsilon$, $d\left(y,\partial
\Omega\right)>\epsilon$, and $d\left(x,y\right)<2^{k}\epsilon$, there exists
a Harnack chain joining $x$ to $y$ of length $Mk$ and such that the diameter of
each ball is bounded from below by $M^{-1}\min\left\{d\left(x,\partial\Omega
\right),d\left(y,\partial\Omega\right)\right\}$.
\end{enumerate}
\end{definition}

The following important property of $NTA$ domains  will be used in the proof of the local comparison theorem. It was established
in \cite{CG1}.

\begin{proposition}
\label{UseChain}Let $Q\in \Omega .$ For any $x,y$ such that $d\left(
x,\partial \Omega \right) ,d\left( y,\Omega \right) >\epsilon ,$ $x,y\in
\Omega \setminus B\left( Q,\frac{\epsilon }{M}\right) $ and $d\left(
x,y\right) \leq C\epsilon ,$ it is possible to choose a Harnack chain $%
\left\{ B_{i}\right\} _{i=1,...,k},$ joining $x$ to $y$, with the
properties: \ 

(i)\ \ \ The length k of the chain depends only on C;

(ii) \ Q$\notin $B$_{i}$ for $i=1,...,k;$

(iii) $\ \widetilde{M}^{-1}\text{diam}B_{i}\leq d\left( B_{i},\partial
\Omega \cup \left\{ Q\right\} \right) \leq \widetilde{M}\text{diam}B_{i},$
where $\widetilde{M}$ depends only on $M$ and on the doubling constant in (\ref{doubling}).
\end{proposition}

Let $D_{T}=\Omega \times \left( 0,T\right) $, with $T>0$. We indicate with $
S_{T}=\partial \Omega \times \left( 0,T\right) $ the lateral boundary of $
D_{T}$, and by $\partial _{p}D_{T}=S_{T}\cup \bar{\Omega}\times \left\{
0\right\} $. For $\left( Q,s\right) \in \partial _{p}D_{T}$ and for $r>0$ we
define 
\begin{equation*}
\Psi _{r}\left( Q,s\right) =\left\{ \left( x,t\right) \in \mathbb{R}
^{n+1}\mid d\left( x,Q\right) <r,\left\vert t-s\right\vert <r^{2}\right\} ,
\end{equation*}
\begin{equation*}
\Psi _{r,Kr}\left( Q,s\right) =\left\{ \left( x,t\right) \in \mathbb{R}
^{n+1}\mid d\left( x,Q\right) <r,\left\vert t-s\right\vert
<K^{2}r^{2}\right\}
\end{equation*}
\begin{equation*}
\Delta _{r}\left( Q,s\right) =\partial _{p}D_{T}\cap \bar{\Psi _{r}}\left(
Q,s\right)
\end{equation*}
\begin{equation*}
\Delta _{r,Kr}\left( Q,s\right) =\partial _{p}D_{T}\cap \bar{\Psi _{r,Kr}}
\left( Q,s\right)
\end{equation*}
\begin{equation*}
\bar{A}_{r}\left( Q,s\right) =\left( A_{r}\left( Q\right) ,s+2r^{2}\right)
,\qquad \underline{A}_{r}\left( Q,s\right) =\left( A_{r}\left( Q\right)
,s-2r^{2}\right) .
\end{equation*}
We call $\Delta _{r}\left( Q,s\right) $ the parabolic surface box of radius 
$r>0$. For $\delta>0$, set 
\begin{equation*}
\Omega ^{\delta }=\left\{ x\in \Omega :\text{dist}\left( x,\partial \Omega
\right) >\delta \right\}
\end{equation*}
\begin{equation*}
D_{T}^{\delta }=\Omega ^{\delta }\times \left( \delta ^{2},T\right)
\end{equation*}

Let $\Gamma\left(x,y\right)=\Gamma\left(y,x\right)$ be the positive fundamental solution of the sub-Laplacian $\sum_{i=1}^{m}X_{i}^{\ast }X_{i}$. The following definition will be needed in this paper.
\begin{definition}
For every $x\in \mathbb{R}^{n}$, and $r>0$, the set 
\begin{displaymath}
B_{X}\left(x,r\right)=\left\{y \in \mathbb{R}^{n}: \Gamma\left(x,y\right)>\frac{r^{2}}{\Lambda\left(x,r\right)}\right\}
\end{displaymath}
will be called the $X$-ball, centered at $x$ with radius $r$.
\end{definition}
The $X$-balls are equivalent to the Carnot-Caratheodory balls : for every $U\subset\mathbb{R}^{n}$, there exists $a>1$, depending on $U$ and $X$, such that 
\begin{equation}
\label{X-balls}
B_{d}\left(x,a^{-1}r\right)\subseteq B_{X}\left(x,r\right)\subseteq B_{d}\left(x,ar\right),
\end{equation} 
for $x\in U$, $0<d\left(x,y\right)\leq R_{o}$, for some $R_{o}$.\\

The following basic estimate was established in \cite{KS}, see also \cite{JSC}.
\begin{theorem}
The fundamental solution $p(x,t;\xi ,\tau )=p(x;\xi ,t-\tau )$ with
singularity at $(\xi ,\tau )$ satisfies the following size estimates: there
exists $M=M\left( X\right) >0$ and for every $k,s\in \mathbb{N}\cup \left\{
0\right\} $, there exists a constant $C=C\left( X,k,s\right) $, such that 
\begin{equation}
\left\vert \frac{\partial {^{k}}}{\partial {t^{k}}}
X_{j_{1}}X_{j_{2}}...X_{j_{s}}p(x,t;\xi ,\tau )\right\vert \leq \frac{C}{
\left( t-\tau \right) ^{s+2k}}\frac{1}{\left\vert B\left( x,\sqrt{t-\tau }
\right) \right\vert }\exp \left( -\frac{Md\left( x,\xi \right) ^{2}}{t-\tau }
\right) ,
\end{equation}
\begin{equation}
p(x,t;\xi ,\tau )=p(x;\xi ,t-\tau )\geq \frac{C^{-1}}{\left\vert B\left( x,
\sqrt{t-\tau }\right) \right\vert }\exp \left( -\frac{M^{-1}d\left( x,\in
\right) ^{2}}{t-\tau }\right) ,
\end{equation}
for every $x,\xi \in \mathbb{R}^{n}$ and any $-\infty <\tau <t<\infty $.
\end{theorem}

As it is pointed out in \cite{FSt}, the existence of Gaussian bounds is equivalent to existence of the Harnack inequality. The following theorem states the Harnack inequality as it was derived in \cite{KS} from the Gaussian bounds.
\begin{theorem}
There is an $M>0$ such that for all $x,y\in \mathbb{R}^{n}$, $s<t$ with $
R=d\left( x,y\right) \vee \left( t-s\right) ^{1/2}$ and all $u>0$ with $Lu=0$
in $\left[ s,s+R^{2}\right] \times B_{d}\left( x,R\right) $: 
\begin{equation}
u\left( y,s\right) \leq u\left( x,t\right) \exp \left( M\left( 1+\frac{
d\left( x,y\right) ^{2}}{t-s}\right) \right)
\end{equation}
\label{Harnack1}
\end{theorem}
We will need the following strong maximum principle to develop Perron's method. 
\begin{theorem}
Let $D=\Omega\times\left(0,T\right)$ where $\Omega\in\mathbb{R}^{n}$ is a connected open set and $T>0$. Suppose that $Lu\geq0$ in $D$. Set $M=\sup_{\bar{D}}u$. Then either $u\left(x,t\right)<M$ for any $\left(x,t\right)\in \bar{D}\setminus \partial_{p} D$ or if $u\left(x_{0},t_{0}\right)= M$, then $u\equiv M$ in $\bar{D}_{t_{0}}$.
\end{theorem}
Only the sketch of the proof will be provided. In his 1969 paper (\cite{Bony}), Bony considered operators of the following form
\begin{displaymath}
\sum^{m}_{i=1}Z^{2}_{i}+Y+a
\end{displaymath}
where the Lie algebra generated by $Z_{1},...,Z_{m}$ and $Y$ generates the whole tangent space. Since $X^{*}_{i}=X_{i}+b_{i}$, for some $b_{i}$,we can write our operator $L$ in the above form. Hence, we can apply Theorem 3.2 in (\cite{Bony}) to our operator $L$. Now, if $u\left(x_{0},t_{0}\right)= M$ the operator $-\partial/\partial t$ forces the maximum to propagate to times $t\leq t_{0}$. This means $u\equiv M$ in $\bar{D}_{t_{0}}$.\\

The results in (\cite{Bony}) also imply that any $\varphi\in C\left(\partial_{p} D_{T},\mathbb{R}\right)$ is resolutive. Then there exists the Perron-Wiener-Brelot solution to the Dirichlet problem
\begin{equation}
Lu=0\text{ in }D_{T},\text{ }u=\varphi \text{ on }\partial _{p}D_{T}
\label{DP}
\end{equation}

\begin{theorem}
Let $\Omega$ be a connected, bounded open set, and $\varphi\in C\left(\partial_{p} D\right)$. Then there exists a unique caloric function $H^{D}_{\varphi}$ which solves (\ref{DP}) in the sense of Perron-Wiener-Brelot. Moreover, $H^{D}_{\varphi}$ satisfies
\begin{equation}
\sup_{D} \left|H^{D}_{\varphi}\right|\leq \sup_{\partial_{p} D} \left|\varphi\right|
\label{PWB}
\end{equation}
\end{theorem}

The previous theorem allows to define the $L-caloric$ measure $\mathrm{d}\omega^{\left(x,t\right)}$ for $D$ evaluated at $\left(x,t\right)\in D$ as the unique probability measure on $\partial_{p} D$ such that for every $\varphi\in C\left(\partial_{p} D\right)$
\begin{displaymath}
H^{D}_{\varphi}\left(x,t\right)=\int_{\partial_{p} D}{\varphi\left(y,s\right)\mathrm{d}\omega^{\left(x,t\right)}\left(y,s\right)}
\end{displaymath}
Another result in Bony's paper tells that a boundary function $\varphi$ is $L$-resolutive if and only if $\varphi\in L^{1}\left(\partial_{p} D_{T},\mathrm{d}\omega^{\left(x,t\right)}\right)$. 

\section{Estimates for $L-caloric$ measure and solutions of $Lu=0$ in $NTA$ domains}

The purpose of this section is to established several basic estimates for $L-caloric$ measure and solutions of $Lu=0$. For instance, in this section one can find the H\"older continuity of solutions near the boundary, Carleson estimate and the Dahlberg's theorem, among others. 
\begin{lemma}
\label{NonVa}
Let $\left( Q,s\right) \in \partial _{p}D_{T}$ and $r>0$
sufficiently small, depending on $r_{0}$. Then, there exists a constant $C>0$%
, depending on $\mathcal{L}$, $M$, and $r_{0}$ such that 
\begin{equation}
\inf_{\Psi _{r}\left( Q,s\right) \cap D_{T}}\omega ^{\left( x,t\right)
}\left( \Delta _{2r}\left( Q,s\right) \right) \geq C.
\end{equation}
\end{lemma}

\begin{proof}
We are going to give the proof for the case $s>0$, the case $s=0$ is treated
similar. By the exterior corkscrew condition we can find a $
\mu=\mu\left(M\right)>0$, such that 
\begin{equation*}
\Psi^{\prime }=B_{d}\left(\bar{Q},\mu r\right)\times
\left(s-4r^2,s+4r^{2}\right)\subset \Psi_{2r}\left(Q,s\right)\setminus D_{T}.
\end{equation*}
for some $\bar{Q}\in\Omega^{c}$. Consider the bottom of this cylinder,
namely $\Delta^{\prime }_{\mu r}=B_{d}\left(\bar{Q},\mu r\right)\times
\left\{s-4r^{2}\right\}$. Recall that $\omega^{\left(x,t\right)}\left(
\Delta_{2r}\right)$ is $1$ on $\Delta_{2r}$ and nonnegative in the rest of $
\bar{\Psi}_{2r}$. On the other hand, if $v\left(x,t\right)=\omega^{\left(x,t
\right)}_{\Psi_{2r}}\left(\Delta^{\prime }_{\mu r}\right)$, where $
\omega^{\left(x,t\right)}_{\Psi_{2r}}$ denotes the caloric measure of $
\Psi_{2r}$, we have $v=0$ in $\partial_{p} \Psi_{2r}\cap D_{T}$ and $v\leq 1$
on $\Delta_{2r}$. By the comparison principle, we end up with 
\begin{equation}
\omega^{\left(x,t\right)}\left(\Delta_{2r}\right)\geq v\left(x,t\right)
\qquad \text{in} \qquad \Psi_{2r}\left(Q,s\right)\cap D_{T},
\end{equation}
which implies that 
\begin{equation*}
\inf_{\Psi_{r}\left(Q,s\right)\cap D_{T}}
\omega^{\left(x,t\right)}\left(\Delta_{2r}\left(Q,s\right)\right)\geq
\inf_{\Psi_{r}\left(Q,s\right)\cap D_{T}} v.
\end{equation*}
Using the maximum principle once more we obtain 
\begin{equation*}
v\left(x,t\right)\geq v^{\prime
}\left(x,t\right)=\omega^{\left(x,t\right)}_{\Psi^{\prime
}_{2r}}\left(\Delta^{\prime }_{\mu r}\right)\qquad \text{in} \qquad
\Psi^{\prime }_{2r}.
\end{equation*}
Now, we can apply the Harnack inequality to $v$ inside $\Psi_{2r}\left(Q,s
\right)$ to obtain 
\begin{equation*}
\inf_{\Psi_{r}\left(Q,s\right)\cap D_{T}} v\geq Cv\left(\bar{Q}
,s-2r^{2}\right)\geq Cv^{\prime }\left(\bar{Q},s-2r^{2}\right),
\end{equation*}
for some constant $C>0$. In order to finish the proof we extend the
function $v^{\prime }$ to a larger cylinder. Consider the cylinder 
\begin{equation*}
\Psi^{\prime \prime }=B_{d}\left(\bar{Q},\mu r\right)\times
\left(s-5r^{2},s+4r^{2}\right).
\end{equation*}
Extend $v^{\prime }$ by the formula 
\begin{equation*}
v^{\prime }\left(x,t\right)=\omega^{\left(x,t\right)}_{\Psi^{\prime \prime
}}\left(\partial_{p} \Psi^{\prime \prime }\cap \left\{t\leq
s-4r^{2}\right\}\right),
\end{equation*}
hence $v^{\prime }\equiv 1$ on $\Psi^{\prime \prime }\cap \left\{t\leq
s-4r^{2}\right\}$. Using the Harnack inequality in $\Psi^{\prime \prime }$ ,
we finally obtain 
\begin{equation*}
v^{\prime }\left(\bar{Q},s-2r^{2}\right)\geq Cv^{\prime }\left(\bar{Q}
,s-4r^{2}\right)=C.
\end{equation*}
\end{proof}
\qed

As a corollary we obtain the Holder continuity at the boundary.

\begin{corollary}
\label{HolBou}Under the assumptions of the previous lemma, let u be a
nonnegative solution of $\mathcal{L}u=0$ which continuously vanishes on $
\Delta _{2r}\left( Q,s\right) $. Then 
\begin{equation}
\sup_{\Psi _{r}\left( Q,s\right) }u\leq \theta \sup_{\Psi _{2r}\left(
Q,s\right) }u,
\end{equation}%
for some constant $\theta \in \left( 0,1\right) $ depending on $L$
, $M$, and $r_{0}$.
\end{corollary}

\begin{proof}
Let $\bar{\omega}^{\left( x,t\right)}$ denote the caloric measure for $\Psi
_{2r}\left( Q,s\right) \cap D_{T}$. Since $\Delta _{2r}\left( Q,s\right) $
lies in the boundary of an $NTA$ domain we can apply the above result to $\bar{\omega}^{\left( x,t\right)}\left(\Delta_{2r}\left(Q,s\right)\right)$.
For $\left( x,t\right) \in \Psi _{r}\left( Q,s\right) \cap D_{T}$, 
\begin{align*}
u\left( x,t\right) & =\int_{\partial _{p}\left( \Psi _{2r}\left( Q,s\right)
\cap D_{T}\right) }u\mathrm{d}\bar{\omega}^{\left( x,t\right)
}=\int_{\partial _{p}\left( \Psi _{2r}\left( Q,s\right) \cap D_{T}\right)
\setminus \Delta _{2r}}u\mathrm{d}\bar{\omega}^{\left( x,t\right) } \\
& \leq \sup_{\Psi _{2r}\left( Q,s\right) \cap D_{T}}u=\left( 1-\bar{\omega}
^{\left( x,t\right) }\left( \Delta _{2r}\right) \right) \sup_{\Psi
_{2r}\left( Q,s\right) \cap D_{T}}u.
\end{align*}
The previous lemma implies the result with $\theta =1-C<1$.
\end{proof}
\qed

\begin{lemma}
\label{Growth}
Let u be a nonnegative solution of Lu=0 in D. \ Let $\left( x,t\right) $ and 
$\left( y,s\right) $ be in $\Omega $, with $t-s>0,$ $\left( t-s\right)
^{1/2}\geq \theta ^{-1}d\left( x,y\right) $ for some $\theta >1.$
Furthermore, suppose that $d\left( x,\partial \Omega \right) >\epsilon ,$ $
d\left( y,\partial \Omega \right) >\epsilon ,$ and $d\left( x,y\right) \leq
C\epsilon ,$~$\left( t-s\right) ^{1/2}\leq C\epsilon ,$ for some $\epsilon
>0.$ Then, there exists a constant $N=N\left( X,\theta ,r_{0},C\right) $
such that%
\begin{equation*}
u\left( y,s\right) \leq Nu\left( x,t\right)
\end{equation*}
\end{lemma}

\begin{proof}
The proof of this lemma is a standard adaptation of that of Lemma 2.2 in \cite{G} and we omit it.
\end{proof}
\qed

The next result is known as the Carleson estimate. The proof is provided for completeness. 

\begin{theorem}
Let $\left(Q,s\right)\in \partial_{p} D_{T}$ and $u$ be a nonnegative
solution of $\mathcal{L}u=0$ in $D_{T}$ that continuously vanishes on $
\Delta_{2r}\left(Q,s\right)$. Then there exists a constant $C>0$, depending
on $\mathcal{L}$, $M$, and $r_{0}$, such that for $r<r_{0}$ and $
\left(x,t\right)\in\Psi_{r}\left(Q,s\right)$, 
\begin{equation}
u\left(x,t\right)\leq Cu\left(\bar{A}_{r}\left(Q_{0},s_{0}\right)\right).
\end{equation}
\end{theorem}
\begin{proof}
We can assume that $u\left(\bar{A}_{r}\left(Q_{0},s_{0}\right)\right)\neq0$. If $u\left(\bar{A}_{r}\left(Q_{0},s_{0}\right)\right)=0$, by the maximum principle $u\equiv0$, since $u\geq 0$ and $Lu=0$. Let 
\begin{displaymath}
v\left(x,t\right)=\frac{u\left(x,t\right)}{u\left(\bar{A}_{r}\left(Q_{0},s_{0}\right)\right)}.
\end{displaymath}
Let $\left(Q,s\right)\in \Delta_{r}\left(Q_{0},s_{0}\right)$ and $\rho>0$ such that $\Psi_{\rho}\left(Q,s\right)\subset \Psi_{2r}\left(Q_{0},s_{0}\right)$. By the H\"older continuity, there exists a constant $C_{1}\geq 2$, depending on $L,M,r_{0}$, such that
\begin{equation}
\label{Carleson1}
\sup_{\Psi_{\rho/C_{1}}\left(Q,s\right)}v\leq \frac{1}{2}\sup_{\Psi_{\rho}\left(Q,s\right)}v.
\end{equation}
Let $k$ be a non-negative integer, then by the Lemma \ref{Growth} there is a constant $C_{2}$, depending only on $C_{1}$, such that if $\left(y,s\right)\in\Psi_{3/2r}\left(Q_{0},s_{0}\right)$, with $\frac{3}{2r^{2}}\geq s-s_{0}$ and $v\left(y,s\right)>C^{k}_{2}v\left(\bar{A}_{r}\left(Q_{0},s_{0}\right)\right)=C^{k}_{2}$, then
\begin{equation}
dist\left(y,\partial \Omega\right)<\frac{r}{C^{k}_{1}}.
\label{Carleson2}
\end{equation}
Fix $K\geq1$ such that $2^{K}>C_{2}$, and set $N=K+5,C=C_{2}$.\\
Claim: $v\left(x,t\right)\leq C$, for all $\left(x,t\right)\in \Psi^{D_{T}}_{r}\left(Q_{0},s_{0}\right)$.
Suppose that the claim is not true. The idea is to construct a sequence of points in $D_{T}$ whose limit is on the lateral boundary $S_{T}$ and on which $v$ grows to infinity. First, there is $\left(y_{1},s_{1}\right)\in \Psi^{D_{T}}_{r}\left(Q_{0},s_{0}\right)$ such that $v\left(y_{1},s_{1}\right)>C$. By (\ref{Carleson2}) we must have that dist$\left(y_{1},\partial \Omega\right)< r/C^{N}_{1}$. Let $\left(Q_{1},s_{1}\right)$ be the point in $S_{T}$ nearest to $\left(y_{1},s_{1}\right)$, then
\begin{displaymath}
d\left(Q_{1},Q_{0}\right)\leq d\left(Q_{1},y_{1}\right)+d\left(y_{1},Q_{0}\right)\leq \frac{r}{C^{N}_{1}}+r\leq \left(\frac{1}{2^{5}}+1\right)r.
\end{displaymath}
If $\left(x,t\right)\in \Psi_{\rho/C^{5}_{1}}\left(Q_{1},s_{1}\right)$, we have 
\begin{displaymath}
d\left(x,Q_{0}\right)\leq d\left(Q_{1},x\right)+d\left(Q_{1},Q_{0}\right)\leq \frac{r}{C^{N}_{1}}+\left(\frac{1}{2^{5}}+1\right)r\leq \left(\frac{1}{2^{4}}+1\right)r< \left(\frac{3}{2}\right)r
\end{displaymath}
and
\begin{displaymath}
\left|t-s_{0}\right|\leq \left|t-s_{1}\right|+\left|s_{1}-s_{0}\right|\leq\left(\frac{1}{2^{10}}+1\right)r^{2}<\frac{9}{4}r^{2}.
\end{displaymath}
We have just proved that 
\begin{equation}
\label{Carleson3}
\Psi_{\rho/C^{5}_{1}}\left(Q_{1},s_{1}\right)\subset \Psi_{2r}\left(Q_{0},s_{0}\right)
\end{equation}
By (\ref{Carleson1}) and the fact that $N=K+5$, 
\begin{displaymath}
\sup_{\Psi_{\rho/C^{5}_{1}}\left(Q_{1},s_{1}\right)}v\geq 2^{M}\sup_{\Psi_{\rho/C^{N}_{1}}\left(Q_{1},s_{1}\right)}v> C_{2}v\left(y_{1},s_{1}\right)>C^{N+1}_{2}.
\end{displaymath}
This implies the existence of $\left(y_{2},s_{2}\right)\in \Psi^{D_{T}}_{r}\left(Q_{1},s_{1}\right)$ such that $v\left(y_{2},s_{2}\right)>C^{N+1}_{2}$. Observe that
\begin{displaymath}
s_{2}-s_{0}\leq\left|s_{2}-s_{1}\right|+\left|s_{1}-s_{0}\right|<\frac{r^{2}}{C^{10}_{1}}+r^{2}<\left(\frac{1}{2^{5}}+1\right)r^{2}<\frac{3}{2}r^{2}.
\end{displaymath}
This means that we can apply (\ref{Carleson2}) to get dist$\left(y_{2},\partial \Omega\right)<r/C^{N+1}_{1}$. As before, let $\left(Q_{2},s_{2}\right)$ be the point in $S_{T}$ nearest to $\left(y_{2},s_{2}\right)$. Then, 
\begin{align*}
d\left(Q_{2},Q_{0}\right) & \leq d\left(Q_{2},y_{2}\right)+d\left(y_{2},Q_{1}\right)+d\left(Q_{1},y_{1}\right)+d\left(Q_(0),y_{1}\right)\\
& \frac{r}{C^{N+1}_{1}}+ \frac{r}{C^{5}_{1}}+\frac{r}{C^{N}_{1}}+r\leq \left(\frac{1}{2^{M+1}}+\frac{1}{2}+1\right)\frac{r}{2^{5}}+r\leq \left(\frac{1}{2^{4}}+1\right)r.
\end{align*}
Choose $\left(x,t\right)\in \Psi_{\rho/C^{5+1}_{1}}\left(Q_{2},s_{2}\right)$, hence
\begin{align*}
d\left(x,Q_{0}\right) & \leq d\left(x,Q_{2}\right)+d\left(Q_{2},Q_{0}\right)\leq \frac{r}{C^{5+1}_{1}}+\left(\frac{1}{2^{4}}+1\right)r
& \leq \left(\frac{1}{2^{3}}+1\right)r<\frac{3}{2}r.
\end{align*}
In similar way we can prove that $\left|t-s_{0}\right|<\frac{9}{4}r^{2}$. We have just proved that 
\begin{equation}
\label{Carleson4}
\Psi_{\rho/C^{5+1}_{1}}\left(Q_{2},s_{2}\right)\subset \Psi_{2r}\left(Q_{0},s_{0}\right)
\end{equation}
Once again by (\ref{Carleson1}) we get 
\begin{displaymath}
\sup_{\Psi_{\rho/C^{5+1}_{1}}\left(Q_{2},s_{2}\right)}v\geq 2^{M}\sup_{\Psi_{\rho/C^{N+1}_{1}}\left(Q_{2},s_{2}\right)}v> C_{2}v\left(y_{1},s_{1}\right)>C^{N+2}_{2}.
\end{displaymath}
We can conclude that there is $\left(y_{3},s_{3}\right)\in \Psi^{D_{T}}_{r}\left(Q_{2},s_{2}\right)$ such that
\begin{displaymath}
v\left(y_{3},s_{3}\right)>C^{N+2}_{2}.
\end{displaymath}
Given that $s_{3}-s_{0}\leq\frac{3}{2}r^{2}$, we get again dist$\left(y_{3},\partial \Omega\right)<\frac{r}{C^{N+2}_{1}}$, and we choose $\left(Q_{3},s_{3}\right)$ the point in $S_{T}$ nearest to $\left(y_{3},s_{3}\right)$. If we keep doing this we will get sequences $\left\{\left(y_{k},s_{k}\right)\right\},\left\{\left(Q_{k},s_{k}\right)\right\}$ with the following properties
\begin{enumerate}
	\item $\Psi_{\rho/C^{5+k-1}_{1}}\left(Q_{k},s_{k}\right)\subset \Psi_{2r}\left(Q_{0},s_{0}\right)$
	\item $\left(y_{k},s_{k}\right)\in \Psi_{\rho/C^{5+k-1}_{1}}\left(Q_{k},s_{k}\right)$
	\item $dist\left(y_{k},\partial \Omega\right)< r/C^{N+k-1}_{1}$
	\item $s_{k}-s_{0}\leq \frac{3}{2}r^{2}$
	\item $v\left(y_{k},s_{k}\right)>C^{N+k-1}_{2}$
\end{enumerate}
The desire sequence is $\left(y_{k},s_{k}\right)$ by (3). We have reached a contradiction since $v$ vanishes in the lateral boundary $S_{T}$. 
\end{proof}
\qed

The previous lemma has the following global version.

\begin{theorem}
Let $\left(Q,s\right)\in \partial_{p} D_{T}$ and $u$ be a nonnegative
solution of $\mathcal{L}u=0$ in $D_{T}$ that continuously vanishes in $
\partial_{p} D_{T}\setminus\Delta_{2r}\left(Q,s\right)$ Then there exists a
constant $C>0$, depending on $\mathcal{L}$, $M$, and $r_{0}$, such that for $
\left(x,t\right)\in D_{T} \setminus\Psi_{r}\left(Q,s\right)$ we have 
\begin{equation}
u\left(x,t\right)\leq Cu\left(\bar{A}_{r}\left(Q,s\right)\right).
\label{VarCar}
\end{equation}
\label{VarCarThe}
\end{theorem}

\begin{proof}
We are providing only the proof for the case $s>0$. The case $s=0$ is
treated in similar way. By the maximum principle it suffices to prove (\ref%
{VarCar}) when $\left(x,t\right)\in \partial_{p} \Psi_{r}\left(Q,s\right)$
and $t>s-4r^{2}$. The first step is to obtain the estimate near the lateral
boundary of $D_{T}$ with the help of Carleson estimate. Near the
lateral boundary we will use both the Carleson estimate and the Harnack
inequality.

We can choose $\delta>0$ small enough so that for 
\begin{equation*}
\left(\bar{Q},\bar{s}\right)\in \partial_{p} \Psi_{r}\left(Q,s\right)\cap
S_{T},
\end{equation*}
$\Psi_{2\delta r}\left(\bar{Q},\bar{s}\right)\cap \Psi_{r/2}\left(Q,s\right)$
and $\bar{s}+2\delta^{2}r<s+2r^{2}$. Then, there exist $C>0$ for which for
all $\left(\bar{Q},\bar{s}\right)\in \partial_{p}
\Psi_{r}\left(Q,s\right)\cap S_{T}$, we have 
\begin{equation*}
u\left(x,t\right)\leq u\left(\bar{A}_{\delta r}\left(\bar{Q},\bar{s}%
\right)\right)
\end{equation*}
with $\left(x,t\right)\in\Psi_{\delta r}\left(\bar{Q},\bar{s}\right)$. By
the scale invariant Harnack inequality, there exist a constant $C>0$ such
that for all 
\begin{equation*}
\left(\bar{Q},\bar{s}\right)\in \partial_{p} \Psi_{r}\left(Q,s\right)\cap
S_{T},
\end{equation*}
we have 
\begin{equation*}
u\left(\bar{A}_{\delta r}\left(\bar{Q},\bar{s}\right)\right)\leq
Cu\left(A_{r}\left(Q,s\right)\right).
\end{equation*}

With the help of the two previous inequalities and a covering argument imply
that $u\left(x,t\right)\leq Cu\left(A_{r}\left(Q,s\right)\right)$ holds on 
\begin{equation*}
\partial_{p} \Psi_{r}\left(Q,s\right)\cap \left\{\left(x,t\right)\mid \text{
dist}\left(x,\partial\Omega\right)\leq cr\right\},
\end{equation*}
where $c>0$ and depends only on the $NTA$ character of $\Omega$. In the
remaining part of $\partial_{p} \Psi_{r}\left(Q,s\right)$ we use Harnack's
Principle to finish the proof.
\end{proof}
\qed

\begin{theorem}
\label{DahlbergTheo}
Let $\left( Q_{0},s_{0}\right) \in S$, then for sufficiently
small $r$, say 
\begin{equation*}
r<\min {\left( r_{0}/2,\sqrt{s}/2,\right)},
\end{equation*}
and each $\left( x,t\right) \in D$ with $s+4a^{2}r^{2}\leq t$ we have 
\begin{align}
\label{Dahlest}
C^{-1}\left\vert B_{d}\left( Q_{0},r\right) \right\vert G\left( x,t;\bar{A}
_{2a^{2}r}\left( Q_{0},s_{0}\right) \right)  & \leq \omega ^{\left( x,t\right) }\left( \Delta
_{r}\left( Q_{0},s_{0}\right) \right)  \\ & \leq C\left\vert B_{d}\left( Q_{0},r\right)
\right\vert G\left( x,t;\underline{A}_{2a^2r}\left( Q_{0},s_{0}\right) \right) \notag
\end{align}
\end{theorem}

\begin{proof}
Fix $\left(x,t\right)\in D$ with $s+4r^{2}\leq t$ and define $g\left(\xi,\tau\right)=G\left(x,t;\xi,\tau\right)$ if $\left(\xi,\tau\right)\in D$ and $g\left(\xi,\tau\right)=0$ for $D^{c}$. For $\left(\xi,\tau\right)\in \mathbb{R}^{n+1}_{+}\setminus \partial_{p} D\cup \left\{\left(x,t\right)\right\}$,

\begin{equation}
g\left(\xi,\tau\right)=p\left(x,t;\xi,\tau\right)-\int_{\partial_{p} D}p\left(Q,s;\xi,\tau\right)d\omega^{\left(x,t\right)}\left(Q,s\right).
\label{repremeasure1}
\end{equation} 

By Fatou's lemma we have 
\begin{displaymath}
\int_{\partial_{p} D}p\left(Q,s;\xi,\tau\right)d\omega^{\left(x,t\right)}\left(Q,s\right)\leq p\left(x,t;\xi,\tau\right)<\infty.
\end{displaymath}

Let $\left\{\xi_{j}\right\}_{j\in \mathbb{N}}$, which converges non-tangentially to $\xi\in \partial \Omega$, so that $d\left(\xi_{j},\xi\right)\leq M d\left(\xi_{j},\partial \Omega\right)$. The Gaussian bounds imply that there is a constant $C>0$ such that for $Q\in\partial\Omega$ and $j\in \mathbb{N}$, $p\left(Q,s;\xi_{j},\tau\right)< C p\left(Q,s;\xi,\tau\right)$. Lebesgue dominated convergence theorem tells that (\ref{repremeasure1}) holds for $\xi\in\partial \Omega$. For a function $\phi\in C^{\infty}_{0}\left(\mathbb{R}^{n+1}_{+}\right)$ with $\phi\left(x,t\right)=0$,
\begin{displaymath}
\int_{D}g\left(\xi,\tau\right)L^{\ast}\phi d\xi d\tau =\int_{\partial D}\phi d\omega^{\left(x,t\right)}\left(Q,s\right) 
\end{displaymath}   
From the results in (\cite{DG}), there exists a function 
\begin{displaymath}
\phi\in C^{\infty}_{o}\left(B_{X}\left(Q,2ar\right)\times\left(s-4a^{2}r^{2},s+4a^{2}r^{2}\right)\right)
\end{displaymath}
with $0\leq \phi \leq 1$, $\phi\equiv 1$ on $B_{X}\left(Q,ar\right)\times\left(s-2a^{2}r^{2},s+2a^{2}r^{2}\right)$ and $\left|L^{\ast}\phi\right|\leq C/r^{2}$.
Using the analogue of Carleson estimate for nonnegative solutions of $L^{\ast
}v=0$, we obtain that there is a constant $C=C\left( M,r_{0}\right) $, such
that 
\begin{equation}
G\left( x,t;\xi ,\tau \right) \leq G\left( x,t;\underline{A}_{2a^{2}r}\left(
Q,s\right) \right) ,
\end{equation}
for each $\left( x,t\right) \in D$ with $t\geq s+4a^{2}r^{2}$ and $\left( \xi
,\tau \right) \in \Psi _{2r}\left( Q,s\right)$.  This gives the right hand side of 
(\ref{Dahlest}). The estimate from below follows the lines of the proof in \cite{FGS} for parabolic equations and therefore we omit it.
\end{proof}
\qed

In similar way we can prove the next theorem. This estimate is very important for the proof of the local comparison.  

\begin{theorem}
\label{Dahlberg-elliptic}
Let $\left( Q_{0},s_{0}\right) \in S$, then for sufficiently
small $r$, say 
\begin{equation*}
r<\min {\left( r_{0}/2,\sqrt{s}/2,\right) },
\end{equation*}
and each $\left( x,t\right) \in D$ with $x\in\Omega\setminus B_{d}\left(Q,ar\right)$ we have 
\begin{equation}
\omega ^{\left( x,t\right) }\left( \Delta_{r}\left( Q_{0},s_{0}\right) \right)  \leq C\left\vert B_{d}\left( Q_{0},r\right)
\right\vert G\left( x,t;\underline{A}_{2a^2r}\left( Q_{0},s_{0}\right) \right) 
\end{equation}
\end{theorem}

The next estimate is crucial to the proof of the boundary backward Harnack inequality, see section 4 below.  
\begin{lemma}
Let u be a nonnegative solution of Lu=0 in D. Let $\left( Q,s\right) \in S$
and $0<r\leq \frac{1}{2}\min {\left( r_{0},\sqrt{s}\right) }$. Then 
\begin{equation*}
u\left( \underline{A}_{r}\left( Q,s\right) \right) \leq Nr^{\gamma
}\inf_{\Psi _{r}^{D}}d^{-\gamma }u
\end{equation*}
with $d=d\left( x\right) \equiv \mathnormal{dist}\left( x,\partial \Omega
\right) $ and $N$,$\gamma $ are positive constants depending only on $\left(
X,M\right) $.
\end{lemma}

\begin{proof}
Let $X=\left( x,t\right) \in \Psi _{r}^{D}\left( Q,s\right) $. Then, if $
d=d\left( x\right) =\mathnormal{dist}\left( x,\partial \Omega \right) $,
there exist $k\in \mathbb{N}$ such that 
\begin{equation*}
d\leq \frac{r}{2^{k}}.
\end{equation*}
Now, let $P\in \partial \Omega $ be such that $d\left( x,P\right) =d\left(
x,\partial \Omega \right) .$ Define $\left( x_{i},s_{i}\right) =\underline{A}
_{2^{i}d}\left( P,t\right) .$ Observe that
\begin{equation*}
d\left( A_{2^{k-1}d}\left( P\right) ,A_{r}\left( Q\right) \right) \leq
3\left( 2^{k}d\right) \leq 3\sqrt{2}\left( s_{k-1}-s+2r\right) ^{1/2}.
\end{equation*}
By Lemma \ref{Growth}, there is a constant $N_{0}=N_{0}\left( X,r_{0},M\right) $
such that
\begin{equation*}
u\left( \underline{A}_{r}\left( Q,s\right) \right) \leq N_{0}u\left(
x_{k-1},s_{k-1}\right) .
\end{equation*}
For $i\leq k-2,$ we have
\begin{equation*}
d\left( A_{2^{i+1}d}\left( P\right) ,A_{2^{i}}\left( P\right) \right) \leq
3\left( 2^{i}d\right) =\sqrt{\frac{3}{2}}\left( s_{i}-s_{i+1}\right) ^{1/2}.
\end{equation*}
Since  $d\left( A_{2^{i+1}d}\left( P\right) ,\Omega \right) \geq
M^{-1}2^{i}d,$ $d\left( A_{2^{i}}\left( P\right) ,\Omega \right) \geq
M^{-1}2^{i}d,$ and 
\begin{equation*}
d\left( A_{2^{i+1}d}\left( P\right) ,A_{2^{i}}\left( P\right) \right) \leq
3M\left( 2^{i}d/M\right) ,
\end{equation*}
there is a constant $N_{1}=N_{1}\left( X,r_{0},M\right) $ such that 
\begin{equation*}
u\left( x_{i+1},s_{i+1}\right) \leq N_{1}u\left( x_{i},s_{i}\right) .
\end{equation*}
Take $N=\max \left\{ N_{0},N_{1}\right\} $ and $\gamma $ such that $
2^{\gamma }=N.$ Then
\begin{equation*}
u\left( \underline{A}_{r}\left( Q,s\right) \right) \leq N^{k}=\left(
2^{k}\right) ^{\gamma }\leq \left( \frac{r}{d}\right) ^{\gamma }u\left(
x,t\right) .
\end{equation*}
Since $\left( x,t\right) \in \Psi _{r}^{D}\left(Q,s\right)$ is arbitrary we are done.
\end{proof}
\qed

\section{Backward Harnack Inequality and the Doubling Condition}
\label{sec 04}
We start this section by showing that a elliptic-type Harnack inequality is implied by the Carleson lemma.
\begin{theorem}
Let u be a nonnegative solution of Lu=0 in a bounded $NTA$ cylinder $D_{T}$
which continuously vanishes on $S_{T}$, and let $0<\delta\leq\frac{1}{2}\min{
\left(r_{0},T\right)}$. Then 
\begin{equation*}
\sup_{D^{\delta}_{T}} u\leq N\inf_{D^{\delta}_{T}} u
\end{equation*}
where $N=N\left(X,\mathnormal{diam}\Omega,T,m,\delta\right)$.
\label{EllipticHar}
\end{theorem}

\begin{proof}
By the continuity of u in $D_{T}^{\delta }$ there exist $\left(
x_{0},t_{0}\right) $ and $\left( x_{1},t{1}\right) $ such that 
\begin{equation*}
u\left( x_{0},t_{0}\right) =\min_{\bar{D}_{T}^{\delta }}u,\quad u\left(
x_{1},t_{1}\right) =\max_{\bar{D}_{T}^{\delta }}.
\end{equation*}
Set $D_{\delta ,T}^{\ast }=\Omega \times \left( \frac{\delta ^{2}}{2},T
\right] $. Since $D_{T}^{\delta }\subset \subset D_{\delta ,T}^{\ast }$ it
is enough to show that 
\begin{equation}
\max_{D_{\delta ,T}^{\ast }}u\leq Nu\left( x_{0},t_{0}\right)  \label{ellhar}
\end{equation}
for some constant $N$. Notice first that by the Harnack principle there is a
constant $N=N\left( X,\mathnormal{diam}\Omega ,T,m,\delta \right) $ such
that 
\begin{equation}
\max_{\Omega ^{\delta /4}\times {\delta ^{2}/2}}u\leq Nu\left(
x_{0},t_{0}\right) .  \label{ellHar}
\end{equation}%
For the points $x\in \Omega $ such that $\mathnormal{dist}\left( x,\partial
\Omega \right) \leq \delta /4$ we will use the Carleson estimate as follows.
Let $Q\in \partial \Omega $ and set $s=\delta ^{2}/2$. By the Carleson
estimate applied to the box $D_{\delta /2}\left( Q,s\right) $, we get that
for all $\left( x,t\right) \in D_{\delta /4}\left( Q,s\right) $, X
\begin{equation}
u\left( x,t\right) \leq N_{1}u\left( \bar{A}_{\delta /4}\left( Q,s\right)
\right) ,  \label{ellHar1}
\end{equation}%
where $N_{1}$ depends on $X,M,r_{0}$. Observe that $D_{\delta /2}\left(
Q,s\right) \subset D_{\delta /2,T}^{\ast }\setminus \bar{D}_{T}^{\delta }$,
and that $s+\delta ^{2}/4=3\delta ^{2}/4$. Hence we can apply the Harnack
inequality to get a constant $N_{2}=N_{2}\left( X,\mathnormal{diam}\Omega
,T,m,\delta \right) $ such that for all $\left( Q,s\right) \in S_{T}$, with $
s=\delta ^{2}/2$, 
\begin{equation}
u\left( \bar{A}_{\delta /4}\left( Q,s\right) \right) \leq N_{2}u\left(
x_{0},t_{0}\right) .  \label{ellHar2}
\end{equation}%
By (\ref{ellHar1}) and (\ref{ellHar2}) we get 
\begin{equation}
u\left( x,t\right) \leq N_{3}u\left( x_{0},t_{0}\right)  \label{ellHar3}
\end{equation}%
for the points $x\in \Omega $ such that $\mathnormal{dist}\left( x,\partial
\Omega \right) \leq \delta /4$. Since $u$ vanishes on $S_{T}$, by (\ref
{ellHar}), (\ref{ellHar3}) and the maximum principle we get (\ref{ellhar}).
\end{proof}
\qed

The following theorem is the backward Harnack inequality which is one of the main results of this paper.
\begin{theorem}
Let $\Omega $ be an $NTA$ domain with parameters $\left( M,r_{0}\right) $, $
u\geq 0$, Lu=0 in $D=\Omega \times \left( 0,\infty \right) $, $u\equiv 0$ in
S. Take $\left( Q,s\right) \in S$, $s\geq \delta _{0}^{2}$, $0<r<\frac{1}{2}
\min \left( r_{0},\delta _{0}\right) $. Then 
\begin{equation}
u\left( \bar{A}_{r}\left( Q,s\right) \right) \leq Nu\left( \underline{A}
_{r}\left( Q,s\right) \right)  \label{BackHar}
\end{equation}
with the constant $N=N\left( X,\mathnormal{diam}\Omega ,T,m,\delta
_{0},r_{0}\right) $
\end{theorem}

\begin{proof}
Let $v\left( x,t\right) =u\left( x,t-s+\delta _{0}^{2}\right) $. Then $
v\left( x,s\right) =u\left( x,\delta _{0}^{2}\right) $. Hence, we can reduce
the proof to the case $s=\delta _{0}^{2}$. Furthermore, we can assume that 
\begin{equation*}
Kr\leq \rho _{0}\equiv \frac{1}{2}\min \left( r_{0},\delta _{0}\right)
\end{equation*}
for some constant $K=K\left( X\right) >6$ that will be specified later. For $
\rho >0$, we define 
\begin{equation*}
\Psi _{\rho }^{D}=D\cap \Psi _{\rho }\left( Q,s\right) ,\quad f\left( \rho
\right) =\rho ^{-\gamma }\sup_{\Psi _{\rho }^{D}}u
\end{equation*}
where $\gamma =\gamma \left( X,m\right) $ is the constant of the previous
lemma. Take 
\begin{equation}
h=\max \left\{ \rho :2r\leq \rho \leq \rho _{0},\quad f\left( \rho \right)
\geq f\left( 2r\right) \right\} .  \label{BackHar1}
\end{equation}
Observe that $\bar{A}_{r}\left( Q,s\right) \in \Psi _{2r}^{D}$, hence 
\begin{equation*}
u\left( \bar{A}_{r}\left( Q,s\right) \right) \leq \sup_{\Psi _{2r}^{D}}u\leq
\left( 2r\right) ^{\gamma }h^{-\gamma }\sup_{D_{h}}u.
\end{equation*}
By the previous lemma, with $r=h$, we obtain 
\begin{equation*}
u\left( \underline{A}_{h}\left( Q,s\right) \right) \leq Nh^{\gamma }
\mathnormal{dist}\left( \underline{A}_{r}\left( Q,s\right) ,\partial \Omega
\right) ^{-\gamma }u\left( \underline{A}_{r}\left( Q,s\right) \right) ,
\end{equation*}
since $\underline{A}_{r}\left( Q,s\right) \in D_{h}$. Given that the
distance from $\underline{A}_{r}\left( Q,s\right) $ to the boundary of $
\Omega $ is proportional to $r$, we get 
\begin{equation}
u\left( \underline{A}_{h}\left( Q,s\right) \right) \leq Nh^{\gamma
}r^{-\gamma }u\left( \underline{A}_{r}\left( Q,s\right) \right) .
\end{equation}
If we can show that 
\begin{equation}
\sup_{\Psi _{h}^{D}}u\leq Nu\left( \underline{A}_{h}\left( Q,s\right)
\right) ,  \label{BackHar2}
\end{equation}
then we will obtain (\ref{BackHar}).\newline
We will divide the proof of (\ref{BackHar2}) in to cases: $Kh>\rho _{0}$ and 
$Kh\leq \rho _{0}$. Suppose that $Kh>\rho _{0}$. By Carleson estimate 
\begin{equation*}
\sup_{\Psi _{h}^{D}}u\leq Nu\left( \bar{A}_{h}\left( Q,s\right) \right) .
\end{equation*}
By the interior elliptic type Harnack inequality we obtain (\ref{BackHar}). 
\newline
Now, for the case $Kh\leq \rho _{0}$ we have 
\begin{equation}
\sup_{\Psi _{Kh}^{D}}u<K^{\gamma }\sup_{\Psi _{h}^{D}}u,  \label{BackHar3}
\end{equation}
since $f\left( Kh\right) <f\left( 2r\right) \leq f\left( h\right) $. Set 
\begin{equation*}
U\equiv \Omega _{Kh}\times \left( s-4h^{2},s+h^{2}\right) ,
\end{equation*}
then 
\begin{equation*}
\Psi _{h}^{D}\subset U\subset \Psi _{Kh}^{D}\subset D,
\end{equation*}
where $\Omega _{h}=\Omega \cap B_{d}\left( Q,h\right) $. Lets break the parabolic boundary of $U$ in three pieces $\Gamma_{0}$, $\Gamma_{1}$, and $\Gamma_{2}$ and write $u$ as 
\begin{equation*}
u\left( x,t\right) =\int_{\partial _{p}U}u\mathrm{d}\omega ^{\left(
x,t\right) }=\int_{\Gamma _{0}}u\mathrm{d}\omega ^{\left( x,t\right)
}+\int_{\Gamma _{1}}u\mathrm{d}\omega ^{\left( x,t\right) }+\int_{\Gamma
_{2}}u\mathrm{d}\omega ^{\left( x,t\right) },
\end{equation*}
where $\Gamma _{0}\equiv S\cap \partial _{p}U$, $\Gamma _{1}\equiv \Omega
_{\left( K-3\right) h}\times \left\{ s-4h^{2}\right\} $, and $\Gamma _{2}$
is the remaining part of $\partial _{p}U$. Given that $u$ vanishes in $
\Gamma _{0}$, 
\begin{equation}
\sup_{\Psi _{h}^{D}}u\leq \sup_{\Gamma _{1}}u+\sup_{\Psi _{h}^{D}}\omega
^{\left( x,t\right) }\left( \Gamma _{2}\right) \cdot \sup_{\Psi _{Kh}^{D}}u.
\label{BackHar4}
\end{equation}
Assume that 
\begin{equation}
\sup_{\Psi _{h}^{D}}\omega ^{\left( x,t\right) }\left( \Gamma _{2}\right)
\leq NK^{Q}e^{-M\left( K-6\right) ^{2}},
\label{BackHar5}
\end{equation}
hence, we have 
\begin{equation*}
\sup_{\Psi _{h}^{D}}u\leq \sup_{\Gamma _{1}}u+NK^{Q}e^{-M\left( K-6\right)
^{2}}K^{\gamma }\sup_{\Psi _{h}^{D}}u\leq \sup_{\Gamma _{1}}u+\frac{1}{2}
\sup_{\Psi _{h}^{D}}u,
\end{equation*}
for some $K=K\left( X\right) >6$, and hence 
\begin{equation*}
\sup_{\Psi _{h}^{D}}u\leq 2\sup_{\Gamma _{1}}u.
\end{equation*}
Now, we only have to show that 
\begin{equation}
\sup_{\Gamma _{1}}u\leq Nu\left( \underline{A}_{h}\left( Q,s\right) \right)
=u\left( A_{h}\left( Q\right) ,s-2h^{2}\right) .  \label{BackHar6}
\end{equation}
Choose $\left( x,s-4h^{2}\right) \in \Gamma _{1}$. If $\mathnormal{dist}
\left( x,\partial \Omega \right) <h$, take $z\in \partial \Omega $ such that 
$d\left(z,x\right)=d\left( x,\partial \Omega \right)$. Then, consider $\left( z,s-5h^{2}\right) $. Since 
\begin{equation*}
\left( x,s-4h^{2}\right) \in \bar{\Psi}_{h}^{D}\left( z,s-5h^{2}\right) ,
\end{equation*}
by the Carleson estimate we obtain 
\begin{equation*}
u\left( x,s-4h^{2}\right) \leq \sup_{\Psi _{h}^{D}\left( z,s-5h^{2}\right)
}u\leq Nu\left( A_{h}\left( z\right) ,s-3h^{2}\right) .
\end{equation*}
Observe that $\left( A_{h}\left( z\right) ,s-3h^{2}\right) \in U$. Now, $
d\left( A_{h}\left(z\right),A_{h}\left(Q\right)\right) <\left( K-1\right) h$. We can apply Lemma (\ref{Growth}) to function $u$, with $\theta=K-1$, $C=M\left(K-1\right)$, and $\epsilon=\frac{h}{M}$ to obtain 
\begin{equation*}
u\left( A_{h}\left( z\right) ,s-3h^{2}\right) \leq Nu\left( \underline{A}
_{h}\left( Q,s\right) \right) ,
\end{equation*}
where $N=N\left( X,M\right) $. Therefore, we have (\ref{BackHar6}) for dist$
\left( x,\partial \Omega \right) <h$. If $d\left( x,\partial \Omega \right)
>h,$ the estimate $u\left( x,t\right) \leq Nu\left( \underline{A}_{h}\left(
Q,s\right) \right) $ follows from the Harnack inequality.

Take $\left( x,t\right) \in \Psi _{h}^{D}$ and $\left( z,\tau \right) \in
\Gamma _{2}$, then $d\left( x,Q\right) <h$, $d\left( z,Q\right) \geq \left(
K-3\right) h$; hence $d\left( z,x\right) \geq \left( K-4\right) h$. For
proving (\ref{BackHar5}), we may assume that $s=4h^{2}$, so that 
\begin{equation*}
\Psi _{h}^{D}=\Omega _{h}\times \left( 3h^{2},5h^{2}\right) ,\quad \Gamma
_{2}\subset \left( \mathbb{R}^{n}\setminus B_{d}\left( Q,\left( K-3\right)
h\right) \right)
\end{equation*}
Assuming that $K>6$, we will compare $u\left( x,t\right) \equiv \omega
^{\left( x,t\right) }\left( \Gamma _{2}\right) $ with the solution $v\left(
x,t\right) $ of the problem 
\begin{equation*}
Lv=0\quad \mathnormal{in}\quad \mathrm{R}^{n}\times \left( 0,5h^{2}\right)
,\quad v\left( x,0\right) =\chi _{\left\{ Kh\geq d\left( x,Q\right) \geq
\left( K-5\right) h\right\} }\left( x\right) .
\end{equation*}
By using the fundamental solution $p\left( x,t;y,s\right) $, we can write 
\begin{equation*}
v\left( x,t\right) =\int_{\left\{ Kh\geq d\left( y,Q\right) \geq \left(
K-5\right) h\right\} }{p\left( x,t;y,0\right) }\mathrm{d}y.
\end{equation*}
If $d\left( x,Q\right) =\left( K-4\right) h$, $0<t<5h^{2}$, we have $
B_{d}\left( x,h\right) \subset \left\{ Kh\geq d\left( y,Q\right) \geq \left(
K-5\right) h\right\} $ and 
\begin{equation*}
v\left( x,t\right) \geq \int_{B_{d}\left( x,h\right) }{p\left(
x,t;y,0\right) }\mathrm{d}y.
\end{equation*}
Suppose that $\sqrt{t}<h$. Then, 
\begin{align}
v\left( x,t\right) & \geq \frac{C^{-1}}{\left\vert B_{d}\left( x,\sqrt{t}
\right) \right\vert }\int_{B_{d}\left( x,h\right) }{\exp \left( \frac{
-M^{-1}d\left( x,y\right) ^{2}}{t}\right) }\mathrm{d}y  \notag \\
& \geq \frac{C^{-1}}{\left\vert B_{d}\left( x,\sqrt{t}\right) \right\vert }
\int_{B_{d}\left( x,\sqrt{t}\right) }{\exp \left( \frac{-M^{-1}d\left(
x,y\right) ^{2}}{t}\right) }\mathrm{d}y  \notag \\
& \geq N.
\end{align}
For $\sqrt{t}\geq h$, we have 
\begin{align}
v\left( x,t\right) & \geq \frac{C^{-1}}{\left\vert B_{d}\left( x,\sqrt{t}
\right) \right\vert }\int_{B_{d}\left( x,h\right) }{\exp \left( \frac{
-M^{-1}d\left( x,y\right) ^{2}}{t}\right) }\mathrm{d}y  \notag \\
& \geq \frac{\left\vert B_{d}\left( x,h\right) \right\vert }{\left\vert
B_{d}\left( x,\sqrt{t}\right) \right\vert }  \notag \\
& \geq N\frac{\left\vert B_{d}\left( x,h\right) \right\vert }{\left\vert
B_{d}\left( x,\sqrt{5}h\right) \right\vert }\geq N,
\end{align}
since $t\in \left( 0,5h^{2}\right) $. Hence, $v\left( x,t\right) \geq N$ for 
$d\left( x,Q\right) =\left( K-4\right) h$ , with $N=N\left( X\right) $. This
implies that $Nv\geq 1\geq u$ on $\left( \Omega \cap \partial B_{\left(
K-4\right) h}\right) \times \left[ 0,5h^{2}\right] ,$ and $Nv\geq 0=u$ on
the remaining part of the parabolic boundary of $\Omega _{\left( K-4\right)
h}\times \left( 0,5h^{2}\right) .$ Since both functions $u$ and $Nv$ satisfy
the same equation $Lu=0$ in $\Omega _{\left( K-4\right) h}\times \left(
0,5h^{2}\right) \supset \Psi _{h}^{D},$ for arbitrary $X=\left( x,t\right)
\in \Psi _{h}^{D}$ we get
\begin{eqnarray*}
u\left( x,t\right) &\leq &Nv\left( x,t\right) =N\int_{\left\{ Kh\geq d\left(
0,y\right) \geq \left( K-5\right) h\right\} }p\left( x,t;y,0\right) dy \\
&\leq &\frac{NC}{\left\vert B_{d}\left( x,\sqrt{3}h\right) \right\vert }
\int_{\left\{ Kh\geq d\left( 0,y\right) \geq \left( K-5\right) h\right\}
}\exp \left( -\frac{Md\left( x,y\right) ^{2}}{5h^{2}}\right) \\
&\leq &\frac{NC}{\left\vert B_{d}\left( x,\sqrt{3}h\right) \right\vert }
\int_{\left\{ 2Kh\geq d\left( x,y\right) \geq \left( K-6\right) h\right\}
}\exp \left( -\frac{Md\left( x,y\right) ^{2}}{5h^{2}}\right)
\end{eqnarray*}
Hence,
\begin{equation*}
u\left( x,t\right) \leq NK^{Q}\exp \left( -M\left( K-6\right) ^{2}\right) .
\end{equation*}
\end{proof}
\qed

The doubling property of the $L-caloric$ measure is a direct consequence of the backward Harnack inequality and
Theorem \ref{DahlbergTheo}. 

\begin{theorem}
There exist a positive constant $C=(X,M,r_{0},\text{diam}$ $\Omega ,T)$
such that for all $(Q,s)\in \partial _{p}D_{T}$ and $0<r\leq \frac{1}{2}\min
\left\{ r_{0},\sqrt{T-s},\sqrt{s}\right\} $ we have  
\begin{equation}
\omega ^{(x,t)}(\Delta _{2r}(Q,s))\leq C\omega ^{(x,t)}(\Delta _{r}(Q,s))
\end{equation}%
with $d\left( x,Q\right) \leq K\left\vert t-s\right\vert ^{1/2}$ and $
\left\vert t-s\right\vert \geq 16r^{2}.$
\end{theorem}

\section{Local and Global Comparison Theorem}
\label{sec 05}

From now on, we assume that $M>100$ and $M>a$, where $a$ is as in (\ref{X-balls})  For $Q\in \partial
\Omega $ and $0<r<\frac{r_{0}}{M}$ we cover the set
\begin{equation*}
F=\partial \Omega \cap \overline{B_{d}\left( Q,\frac{3M}{4}r\right)
\setminus B_{d}\left( Q,\frac{M}{4}r\right) }
\end{equation*}%
by $N_{0}$ balls $B_{d}\left( Q_{i},\frac{r}{T}\right) ,$ with $Q_{i}\in F$
and $l\geq 2$ suitably chosen.\ The balls $B_{d}\left( Q_{i},\frac{r}{T}
\right) $ can be taken so that $B_{d}\left( Q_{i},\frac{r}{100T}\right) $
will be disjoint. This fact, the interior corkscrew condition and the doubling property implies
that the number $N_{0}$ is independent of $r.$ In similar way, we can cover $
\left( s-\frac{M^{2}r^{2}}{4},s+\frac{M^{2}r^{2}}{4}\right) $ by $%
N_{1}=N_{1}\left( l\right) $ intervals $I_{\frac{r}{l}}\left( s_{i}\right) $
of length $\frac{r^{2}}{l^{2}}.$ Set 
\begin{equation*}
H\left( x,t\right) =\sum_{j=1}^{N_{1}}\sum_{i=1}^{N_{0}}\omega ^{\left(
x,t\right) }\left( \Delta _{r}\left( Q_{i},s_{j}\right) \right) +\left\vert
B_{d}\left( x,r\right) \right\vert G\left( x,t;A_{Mr}\left( Q\right)
,s-4M^{2}r^{2}\right)
\end{equation*}

\begin{lemma}
For $l$ sufficiently large, there exists a constant $C=C\left( M\right) >0$
such that for $\left( x,t\right) \in D_{T}\cap \partial _{p}\Psi _{\frac{Mr}{
2}}\left( Q,s\right)$ 
\begin{equation*}
H\left( x,t\right) \geq C.
\end{equation*}
\end{lemma}

\begin{proof}
By Lemma \ref{NonVa} and Corollary \ref{HolBou} we obtain for $\left(
x,t\right) \in \Psi _{\frac{2r}{l}}\left( Q_{i},s_{i}\right) \cap D_{T}$ 
\begin{equation*}
\omega ^{\left( x,t\right) }\left( \Delta _{r}\left( Q_{i},s_{i}\right)
\right) \geq \frac{1}{2}.
\end{equation*}
Let $l=l\left( M\right) $ be the smallest positive number for which the
above estimate holds for every $i\in \left\{ 1,...,N_{0}\right\} $ and $
j=\left\{ 1,...,N_{1}\right\} .$ We have thus proved the estimate when $x\in
\cup _{i=1}^{N_{0}}B_{d}\left( Q_{i},\frac{2r}{l}\right) .$ It is not difficult to see that if 
\begin{equation*}
x\in A\overset{\mathrm{def}}{=}\Omega \cap \partial B_{d}\left( Q,\frac{Mr}{2
}\right) \setminus \cup _{i=1}^{N_{0}}B_{d}\left( Q_{i},\frac{2r}{l}\right) ,
\end{equation*}
then 
\begin{equation*}
d\left( x,\partial \Omega \right) \geq \frac{r}{l}.
\end{equation*}
Now, by the Harnack inequality applied to the function $G\left( x,t;\text{
\textperiodcentered },\text{\textperiodcentered }\right) $ we obtain
\begin{equation*}
G\left( x,t;A_{Mr}\left( Q\right) ,s-4M^{2}r^{2}\right) \geq CG\left(
x,t;A_{Mr}\left( Q\right) ,s-2M^{2}r^{2}\right) .
\end{equation*}
The boundary backward Harnack inequality gives
\begin{equation*}
G\left( x,t;A_{Mr}\left( Q\right) ,s-2M^{2}r^{2}\right) \geq C_{M}G\left(
x,t+5M^{2}r^{2};A_{Mr}\left( Q\right) ,s-2M^{2}r^{2}\right)
\end{equation*}
since $r$ is small enough. By Theorem (\ref{DahlbergTheo}) and Lemma \ref{Growth} we obtain
\begin{equation*}
G\left( x,t;A_{Mr}\left( Q\right) ,s-4M^{2}r^{2}\right) \geq \frac{C}{
\left\vert B_{d}\left( x,r\right) \right\vert }\omega ^{\left( A_{\frac{r}{2}
}\left( Q\right) ,s+4M^{2}r^{2}\right) }\left( \Delta _{Mr}\left( Q,s\right)
\right) .
\end{equation*}
The conclusion follows from Lemma \ref{NonVa}.
\end{proof}
\qed
\begin{theorem}
Let $(Q,s)\in S_{T}$ and $u,v$ be two nonnegative solutions of $Lu=0$ in $
\Psi _{Mr}^{D}(Q,s)$ vanishing continuously on $\Delta _{Mr}(Q,s).$ Then
there exists a constant $C=C(X,M,r_{0})$ such that for $r<\frac{1}{M}\min
(r_{0},\sqrt{s},\sqrt{T-s}),$ and $(x,t)\in \Psi _{\frac{r}{4aM}}^{D}$ we
have 
\begin{equation}
\frac{u(x,t)}{v(x,t)}\leq C\frac{u(A_{Mr}(Q),s+4M^{2}r^{2})}{%
v(A_{Mr}(Q),s-4M^{2}r^{2})}
\end{equation}
\label{localcomp}
\end{theorem}

\begin{proof}
By the Carleson estimate and the previous lemma 
\begin{equation}
u\left( x,t\right) \leq Cu\left( \overline{A}_{Mr}\left( Q,s\right) \right)
H\left( x,t\right)  \label{LC1}
\end{equation}
for $\left( x,t\right) \in D_{T}\cap \partial _{p}\Psi _{\frac{Mr}{2}}\left(
Q,s\right)$. Let 
\begin{displaymath}
H^{*}\left( x,t\right) =\sum_{j=1}^{N_{1}}\sum_{i=1}^{N_{0}}\omega ^{\left(
x,t\right) }\left( \Delta _{r}\left( Q_{i},s_{j}\right) \right) +\left\vert
B_{d}\left( Q,r\right) \right\vert G\left( x,t;A_{Mr}\left( Q\right)
,s-4M^{2}r^{2}\right)
\end{displaymath}
Observe that for every $x\in D_{T}\cap \Psi _{\frac{Mr}{2}}\left( Q,s\right)$
\begin{equation}
C^{-1}H\left(x,t\right)\geq H^{*}\left(x,t\right) \geq H\left(x,t\right)
\label{LCL} 
\end{equation}
where $C>0$ depends on the constant in (\ref{doubling}). Since $H^{*}$ is $L$-superparabolic in $x\in D_{T}\cap \Psi _{\frac{Mr}{2}}\left( Q,s\right)$, by (\ref{LC1}) and (\ref{LCL}),  
we can conclude that (\ref{LC1}) holds in $
D_{T}\cap \Psi _{\frac{Mr}{2}}\left( Q,s\right) .$ On the other hand, set $
Q_{r}=A_{Mr}\left( Q\right) $, $s_{r}=s-4M^{2}r^{2}.$ Then, there is a $
\delta =\delta \left( M\right) >0$ such that
\begin{equation*}
\Phi _{r}=\left\{ \left( x,t\right) :d\left( x,Q_{r}\right) <\delta r,s_{r}+
\frac{\delta ^{2}r^{2}}{2}<t<s_{r}+\delta ^{2}r^{2}\right\}
\end{equation*}
is contained in $D_{T}$ and $s_{r}+\delta ^{2}r^{2}<s-\frac{r^{2}}{4M^{2}}.$
For $\left( x,t\right) \in \partial \Phi _{r},$ the Gaussian bounds imply
\begin{equation*}
\left\vert B_{d}\left( x,r\right) \right\vert G\left( x,t;A_{Mr}\left(
Q\right) ,s-4M^{2}r^{2}\right) \leq C\left( M\right)
\end{equation*}
Furthermore, by the Harnack inequality and the Harnack chain condition,
\begin{equation*}
v\left( A_{Mr}\left( Q\right) ,s-4M^{2}r^{2}\right) \leq Cv\left( x,t\right)
\end{equation*}
for $\left( x,t\right) \in \partial \Phi _{r}.$ Hence, for $t\geq s-\frac{
r^{2}}{4M^{2}},$
\begin{equation}
\left\vert B_{d}\left( x,r\right) \right\vert G\left( x,t;\underline{A}%
_{Mr}\left( Q,s\right) \right) v\left( \underline{A}_{Mr}\left( Q,s\right)
\right) \leq Cv\left( x,t\right)  \label{LC2}
\end{equation}

In order to finish the proof we need to prove that 
\begin{equation*}
H(x,t)\leq C\left\vert B_{d}\left( x,r\right) \right\vert G\left(
x,t;A_{Mr}\left( Q\right) ,s-4M^{2}r^{2}\right)
\end{equation*}
in $D_{T}\cap \Psi _{\frac{r}{4aM}}\left( Q,s\right) .$ Observe that 
\begin{equation}
B_{d}\left( Q,\frac{r}{4aM}\right) \cap \Omega \subset \Omega \setminus \cup
_{i=1}^{N}B_{d}\left( Q_{i},2r\right) .  \label{LC3}
\end{equation}
By Theorem \ref{Dahlberg-elliptic}, (\ref{LC3}) gives
\begin{equation*}
\omega ^{\left( x,t\right) }\left( \Delta _{r}\left( Q_{i},s_{j}\right)
\right) \leq CG\left\vert B_{d}\left( x,r\right) \right\vert G\left(
x,t;A_{2a^{2}r}\left( Q_{i}\right) ,s_{j}-4a^{2}r^{2}\right)
\end{equation*}
 and the fact that $t>s-\frac{r^{2}}{16\left(aM\right)^{2}}.$ Harnack
inequality and proposition \ref{UseChain} implies 
\begin{equation*}
G\left( x,t;A_{2a^{2}r}\left( Q_{i}\right) ,s_{j}-4a^{2}r^{2}\right) \leq CG\left(
x,t;A_{r}\left( Q\right) ,s-4M^{2}r^{2}\right) .
\end{equation*}
\end{proof}
\qed

The following global comparison theorem is a consequence of Theorems (\ref{EllipticHar}) and (\ref{localcomp}), see \cite{FGS} for details.

\begin{theorem}
Let  $u,v$ be two nonnegative solutions of $Lu=0$ in $D_{+}$ which continuously vanish on $S_{+}$. Then for $0<\delta<\frac{1}{2a}\min\left(r_{o},\sqrt{T}\right)$ there exists a positive constant $C=C\left(X,\mathnormal{diam}\Omega,T,m,\delta\right)$ such that
\begin{displaymath}
v\left(x_{o},T\right)u\left(x,t\right)\leq C u\left(x_{o},T\right)v\left(x,t\right)
\end{displaymath}  
for all $\left(x,t\right)\in \Omega\times \left(2\delta^{2},T-\delta^{2}\right)$, where $x_{o}\in \Omega$ is fixed.  
\end{theorem}

\begin{acknowledgements}
The author wishes to thank Prof. Nicola Garofalo for  his constant interest and the many helpful conversations on sub-elliptic and parabolic equations.
\end{acknowledgements}



\end{document}